\documentclass[12pt]{amsart}
\usepackage{a4wide}
\usepackage{amsmath}
\usepackage{amsfonts}
\usepackage{amssymb}
\usepackage{mathrsfs}
\usepackage{pb-diagram}
\usepackage{color}
\usepackage{epstopdf}
\usepackage{enumerate}
\usepackage{graphicx}
\usepackage{amscd,amsthm,curves,enumerate,latexsym}
\usepackage{bbm}
\usepackage{tikz}
\usepackage{cleveref}
\usepackage[all]{xypic}
\usepackage{epstopdf}
\newtheorem{lemma}{Lemma}[section]
\newtheorem{remark}[lemma]{Remark}
\newtheorem{theorem}[lemma]{Theorem}
\newtheorem{corollary}[lemma]{Corollary}
\newtheorem{conjecture}[lemma]{Conjecture}

\newtheorem{proposition}[lemma]{Proposition}

\renewcommand{\subset}{\subseteq}

\begin{document}

	\title[Parabolic recursions for Kazhdan-Lusztig polynomials]{Parabolic recursions for Kazhdan-Lusztig polynomials and the hypercube decomposition}
	
	
	\begin{abstract}
	We employ general parabolic recursion methods to demonstrate the recently devised hypercube formula for Kazhdan-Lusztig polynomials of $S_n$, and establish its generalization to the full setting of a finite Coxeter system through algebraic proof.

We introduce procedures for positive decompositions of $q$-derived Kazhdan-Lusztig polynomials within this setting, that utilize classical Hecke algebra positivity phenomena of Dyer-Lehrer and Grojnowski-Haiman. This leads to a distinct algorithmic approach to the subject, based on induction from a parabolic subgroup.

We propose suitable weak variants of the combinatorial invariance conjecture and verify their validity for permutation groups.

	\end{abstract}

	\author{Maxim Gurevich}
	\address{Department of Mathematics, Technion -- Israel Institute of Technology, Haifa, Israel.}
	\email{maxg@technion.ac.il}
	
	\author{Chuijia Wang}
	\address{Department of Mathematics, Technion -- Israel Institute of Technology, Haifa, Israel.}
	\email{wangcj@campus.technion.ac.il}
	
	\date{\today}
	
	\maketitle
	
	\section{Introduction}

The Kazhdan-Lusztig polynomials are integral invariants associated with each pair of elements in a Coxeter group. These polynomials are renowned for their ability to convey intricate details on the symmetric structures in which the group appears, such as multiplicities of Lie group representations, dimensions of cohomology spaces of complex varieties, and transition matrices between distinguished bases of Hecke algebras. Given their importance, it can be argued that finding effective methods to access the information encoded in them and provide a concrete algorithmic description of the polynomials is crucial for gaining a fundamental understanding of symmetry.

In a recent work \cite{BBDVW21,BBDVW-nat}, artificial intelligence methods were used to tackle this task for the case of a Kazhdan-Lusztig polynomial $P_{\sigma,\omega}$ attached to a pair of permutations $\sigma,\omega\in S_n$ in a symmetric group. The researchers produced a novel formula, which they then proved using machinery of equivariant intersection cohomology. The formula deals with so-called $q$-derived Kazhdan-Lusztig polynomials $P^\partial_{\sigma,\omega}$, out of which the value of $P_{\sigma,\omega}$ is easily extracted.

Let us state an ad-hoc version of it suitable for the subsequent discussion. We write $\mathbb{Z}[q]^+$ for the semiring of integer polynomials with non-negative coefficients in the formal variable $q$.

\begin{theorem}\cite[Theorem 3.7]{BBDVW21}\label{thm:wilmson}
  For $\sigma,\omega\in S_n$, there exist polynomials $Q_{\sigma,\omega}\in \mathbb{Z}[q]^+$ and $I_{\sigma,\omega,\zeta}\in\mathbb{Z}[q]^+$, for all $\zeta\in S_{n-1}$, such that
  \[
  P^{\partial}_{\sigma,\omega} = Q_{\sigma,\omega} + \sum_{\zeta\in S_{n-1}} I_{\sigma,\omega, \zeta}\;
  \]
  holds.

The values of each of $I_{\sigma,\omega,\zeta}$ are linearly computable out of the values of Kazhdan-Lusztig polynomials for the group $S_{n-1}$ and values of $P_{\kappa,\omega}$ for $\sigma < \kappa\leq \omega$ in the Bruhat order on $S_n$.

The value of $Q_{\sigma,\omega}$ is combinatorially computable out of data supplied by the Bruhat order on the interval $\{\kappa\in S_n\;:\; \sigma\leq \kappa\leq \omega\}$, the location of the coset $S_{n-1}\cdot \sigma$ in that interval and values of $P_{\kappa,\omega}$ as above.

\end{theorem}

The formula in \Cref{thm:wilmson} is referred to as the \textit{hypercube decomposition}, because of the nature of the algorithm producing $Q_{\sigma,\omega}$. It was further conjectured \cite[Conjecture 3.8]{BBDVW21} that the input required for that latter computation does not require the knowledge of the location of the coset inside the Bruhat interval. This property would then imply the celebrated Combinatorial Invariance Conjecture (see discussion in \Cref{sect:cic}).

\subsection{Main results}

Our work has the combined goal of generalizing the formula of \Cref{thm:wilmson} to the case of a finite Coxeter group $W$ in place of $S_n$, while explaining and reproving the original permutations case in terms native to the combinatorics of Hecke algebras.

We provide a general recursion procedure for a Kazhdan-Lusztig polynomial attached to a pair $\sigma,\omega\in W$, in terms of Kazhdan-Lusztig polynomials attached to elements of a chosen \textit{parabolic subgroup} $W_J < W$.

\begin{theorem}\label{thm:A}
Given a finite Coxeter group $W$ with a set of Coxeter generators $S$, for each choice of a subset $J\subset S$, there is a collection of Laurent polynomials $R_{\kappa_1,\kappa_2,J}\in \mathbb{Z}[q^{1/2},q^{-1/2}]$, for all $\kappa_1,\kappa_2\in W$, such that the following holds.

\begin{enumerate}
\item
For all $\sigma,\omega\in W$, there is a decomposition
\[
P^{\partial}_{\sigma,\omega} = Q^J_{\sigma,\omega} + I^J_{\sigma,\omega}\;,
\]
where $Q^J_{\sigma,\omega}, I^J_{\sigma,\omega}\in \mathbb{Z}[q]^+$, so that the identity
\[
Q^J_{\sigma,\omega} =  \frac{q^{\ell(\omega)-\ell(\sigma)}}{1-q} d\left( \sum_{\sigma < \kappa\leq \omega} R_{\sigma,\kappa,J} P_{\kappa, \omega}\right)\;
\]
is satisfied.

Here, $\ell$ is the length function on $W$ and $d$ is the ring involution on Laurent polynomials given by $d(q) = q^{-1}$.

\item\label{it:intr2}
The polynomial $ I^J_{\sigma,\omega}$ is decomposed into a sum
\[
I^J_{\sigma,\omega} = \sum_{\kappa\in W_J\;:\; \sigma < \kappa\,\cdot\,{}^J\sigma \leq \omega } P^{\partial}_{\sigma_J,\kappa} \gamma'_{\kappa}\;,
\]
with $\gamma'_\kappa\in \mathbb{Z}[q]^+$, for all $\kappa\in W_J= \langle J\rangle$ (the parabolic subgroup generated by $J$).

Here, $\sigma = \sigma_J {}^J \sigma$ with $\sigma_J \in W_J$  and ${}^J\sigma\in W$ is the minimal length representative of the coset $W_J \sigma$.
\item
The tuple of polynomials $\{\gamma'_\kappa\}_{\kappa\in W_J}$ is the solution to the linear system
\[
	\left( P_{\zeta, \kappa}\right)_{\zeta,\kappa\in W_J}\cdot (\gamma'_{\kappa})_{\kappa\in W_J} = (P_{\zeta\,\cdot\,{}^J\sigma, \omega})_{\zeta\in W_J}\;,
\]
when written in matrix notation.

\end{enumerate}
\end{theorem}

The Laurent polynomials $R_{\sigma,\omega,J}$, that we refer to as $J$-relative $R$-polynomials, play a crucial role in our work, as evident from the above theorem statement. While we have not come across any explicit study of these polynomials in previous literature, we explain further below how they arise naturally in the Dyer-Lehrer theory for Hecke algebras.

In \Cref{prop:recJR}, we provide a parabolic recursion formula that utilizes $R$-polynomials instead of Kazhdan-Lusztig polynomials, which is a common approach in the Kazhdan-Lusztig theory. Interestingly, this form of recursion is easier to state but fails to capture the positivity properties of Theorem 1.2.

We anticipate a forthcoming work of Brenti-Marietti that will further explore the applications of this form of alternative recursion.

Returning to the case of permutation groups, we offer a new proof for \Cref{thm:wilmson} that relies solely\footnote{The positivity parts of the theorem statement are not deduced algebraically. Instead, we exhibit their deduction from earlier positivity phenomena \cite{DL90,GH07} in Hecke algebras that were originally proved in a geometric manner.} on Hecke-algebraic techniques.

\begin{theorem}\label{thm:B}
Let the group $S_n$ be taken with the standard convention of Coxeter generators being the simple transpositions. Let $J = \{(1,2),\ldots,(n-2,n-1)\} \subset S_n$ be a fixed subset of Coxeter generators.

For $\sigma,\omega\in S_n$, let $\mathcal{P}(\sigma,\omega)$ be the collection of subsets $B\subset \{1,\ldots, n-1\}$, for which the two following conditions are satisfied:
\begin{enumerate}
  \item For all $i\in B$, $\sigma < (i,n)\sigma$ holds in the Bruhat order,

  \item $\omega$ is the supremum of the set $\{\sigma\}\cup\{(i,n)\sigma\}_{i\in B}$ in the Bruhat order.
\end{enumerate}

Then,
\[
R_{\sigma,\omega,J} = \sum_{B\in \mathcal{P}(\sigma,\omega)} (q-1)^{|B|}
\]
holds.

In particular, the decomposition of \Cref{thm:A} coincides with the hypercube decomposition of \Cref{thm:wilmson} in that case.

\end{theorem}

In the final section of our work, we aim to characterize the role of parabolic recursion in the pursuit of combinatorial invariance.

It has long been believed (see the review \cite{Bre04}, or the discussion in \cite{BBDVW21} for a recount of recent progress) that the value of a Kazhdan-Lusztig polynomial $P_{\sigma,\omega}$ should be extractable out of a ``blind look" at the Bruhat order structure on the interval of elements $\zeta\in W$ in between $\sigma\leq \zeta \leq \omega$.

To provide possible milestones towards this classical Combinatorial Invariance Conjecture, we propose two weaker variants, both of which were implicitly achieved as corollaries of \Cref{thm:wilmson} in the case of symmetric groups upon retrospective analysis.

\Cref{conj:R} proposes that for a finite Coxeter system $(W,S)$, there exists a subset $J\subsetneq S$ such that the values of $R_{\sigma,\omega,J}$ can be obtained from the Bruhat graph on the interval between $\sigma$ and $\omega$, along with information on which edges in the graph come from elements in the subgroup $W_J$.

\Cref{conj:weak} suggests that a canonical coloring with at most $|S|$ colors can be assigned to each Bruhat interval graph so that the values of Kazhdan-Lusztig polynomials can be extracted from the colored directed graph.

The following essentially elaborates on a central theme in \cite{BBDVW21}.

\begin{theorem}[\Cref{thm:sncic}]
Both weak combinatorial invariance conjectures as suggested in \Cref{conj:R} and \Cref{conj:weak} are valid in the case of permutation groups.
\end{theorem}

\subsection{Methods of proof}
Our main line of study revisits the theory of Dyer-Lehrer \cite{DL90} for Hecke algebras.

Recall that Kazhdan-Lusztig theory, in its basic appearance as in \cite{KL79}, deals with the Hecke algebra $\mathcal{H}(W)$, associated with a Coxeter system $(W,S)$, and its two $W$-labelled bases, the standard $\{h_\omega\}_{\omega\in W}$ and the canonical $\{c_\omega\}_{\omega\in W}$.

In our notation, we consider, for each given element $\tau\in W$, the twisted basis $\{h_{\tau}^{-1} h_{\tau\omega}\}_{\omega\in W}$ for $\mathcal{H}(W)$. A central early finding of Dyer-Lehrer asserts that the transition matrix between any of the twisted bases and the canonical basis has non-negative coefficients.

By utilizing identities that follow out of consecutive changes of basis, we can employ this positivity result, along with further combinatorial analysis of $R$-polynomials by Dyer, in order to explicate structured decompositions of $q$-derived Kazhdan-Lusztig polynomials.

In \Cref{prop:posIQ} and \Cref{cor:ppos}, we thus construct families of polynomials $I^\tau_{\sigma,\omega}\in \mathbb{Z}[q]^+$, for $\sigma,\omega,\tau\in W$ (here $W$ is finite), so that the symmetry
\begin{equation}\label{eq:intro}
P^{\partial}_{\sigma,\omega} = I^\tau_{\sigma,\omega}  + q^{\ell(\omega)-\ell(\sigma)-1} d(I^{\omega_0\tau}_{\sigma,\omega})\;,
\end{equation}
holds, where $\omega_0\in W$ is the longest element in the group.

In order to attain the parabolic recursion formulas of \Cref{thm:A}, we specialize the decomposition of \eqref{eq:intro} to the case of $\tau = \omega_0^J$, where $\omega_0^J\in W_J$ is the longest element in the parabolic subgroup.

The $J$-relative $R$-polynomials are then defined as the transition constants
\[
h_{\omega_0\omega_0^J} h_\omega = \sum_{\sigma\in W} q^{\frac{\ell(\sigma)-\ell(\omega)}2} R_{\sigma,\omega,J} h_{\omega_0\omega_0^J\sigma} \;,
\]
in between bases in the $\mathbb{Z}[q^{\pm 1/2}]$-algebra $\mathcal{H}(W)$.

A second key aspect to highlight is the mechanism underlying the positive decomposition outlined in \Cref{thm:A}\eqref{it:intr2}. The polynomials $\gamma'_\kappa$, for $\kappa\in W_J$, present a particular instance of another positivity phenomenon that has been previously studied in Hecke algebras.

Grojnowski-Haiman \cite{GH07} introduced a $J$-hybrid basis for $\mathcal{H}(W)$, consisting of products of the form $c_{\omega_J}h_{{}^J \omega}$, for each $\omega = \omega_J {}^J \omega \in W$, when written as in \Cref{thm:A}. Their main finding was that expanding a canonical basis element $c_\omega$ on the $J$-hybrid basis yet again results in non-negative coefficients. In \Cref{prop:inductive}, we demonstrate that these coefficients are precisely the constants required for our decomposition.

\subsection{Further discussion}

In addition to the conjectures we propose in \Cref{sect:cic}, our results suggest several natural issues for future study that we consider important to note.

Firstly, it would be interesting to explore whether the decomposition in \Cref{thm:A} has a direct geometric interpretation, similar to the one originally used to prove \Cref{thm:wilmson}. In \cite{BBDVW21}, $q$-derived Kazhdan-Lusztig polynomials are presented as Poincar\'{e} polynomials of certain projective varieties related to Schubert geometry, and the hypercube decomposition becomes a decategorification of an identity in the derived category of sheaves on that variety. Can further cases of parabolic recursion be obtained in this manner?

Next, beyond the context of combinatorial invariance as in \Cref{conj:R}, it is worth contemplating whether $J$-relative $R$-polynomials may admit formulas as elegant as in \Cref{thm:B} in cases other than those of \Cref{sect:Sn}. In other words, how might a hypercube look like outside of Lie type $A$?

On a related note, machine verification that originated in \cite{BBDVW-nat} demonstrated that the decomposition of \Cref{thm:wilmson} arises in a more general context than that of a coset of a parabolic subgroup in the Coxeter group. It would be interesting to search for similar graph-theoretic axiomatizations that lead to decompositions analogous to \Cref{thm:A} for other finite Coxeter systems.

Finally, we briefly mention a theme that served as the initial motivation for this work. One Lie-theoretic appearance of $S_n$-Kazhdan-Lusztig polynomials is the graded decomposition numbers in the representation theory of type $A$ quiver Hecke algebras (see, for example, \cite{BK09}). The ungraded variant descends into the $p$-adic Kazhdan-Lusztig conjectures of \cite{Z81}, while its decategorified variant amounts to describing the transition between PBW and canonical bases in the quantum group $U_q(\mathfrak{sl}_N)^+$  \cite{L90}.

The relevance of this setting to the hypercube decomposition is due to the fact that proper standard modules for quiver Hecke algebra, whose decomposition into simple constituents is given by $S_n$-Kazhdan-Lusztig polynomials, are constructed as convolution products of $n$ cuspidal\footnote{Not to be confused with the $p$-adic notion of cuspidality.} modules in the sense of \cite{KR11}. Thus, the functor of convoluting with an additional cuspidal module serves as an induction step between the spectral theory of $S_n$ to that of its parabolic subgroup $S_{n-1}$.

In that sense, a categorification of \Cref{thm:wilmson} is possible and will be carried out in a forthcoming work. However, we raise the question of whether other Lie-theoretic settings exist that could provide the categorical foundation for the parabolic recursion of \Cref{thm:A}.

\subsection{Acknowledgements}

The first author would like to thank the organizers of the Oberwolfach August 2022 workshop ``Character Theory and Categorification" for providing a platform to present research related to this work, which resulted in valuable feedback.

We thank Geordie Williamson for sharing his insights on the topic, that gave much of the initial impetus to this project.

The first author is thankful to Martina Lanini for valuable discussions and for hosting his visit at the Tor Vergata University in Rome. We express our gratitude to Francesco Brenti and Mario Marietti for sharing their views on the emerging hypercube theme, and to Gaston Burrull for exposing us to essential background materials.

We extend special appreciation to Nicol\'{a}s Libedinsky for directing us to the crucial observation on the relevance of the work of Grojnowski-Haiman to the subject.

This research is supported by the Israel Science Foundation (Grant Number: 737/20).
	
	\section{Preliminaries}

	

    \subsection{Hecke algebras and Kazhdan-Lusztig polynomials}
	Let $\mathcal{L}:=\mathbb{Z}[v,v^{-1}]$ be the ring of Laurent polynomials in a formal variable $v$, and $\mathcal{L}^+\subset \mathcal{L}$ its sub-semiring of Laurent polynomials with positive coefficients.

As customary, we fix the notation $\alpha = v^{-1}-v, q=v^{-2}\in \mathcal{L}$ throughout this work. We write $d:\mathcal{L}\to\mathcal{L}$ for the ring involution given by $d(v) = v^{-1}$. Clearly, $d(\alpha) = -\alpha$.

Let $(W,S)$ be a finite Coxeter system (that is, the group $W$ is assumed to be finite).

We write $\ell: W \to \mathbb{Z}_{\geq0}$ for its length function, $\omega_0 \in W$ for its longest element, and $e\in W$ for the identity element. We also equip $W$ with Bruhat partial order relative to $S$, denoted as $\leq$.

Let $\mathcal{H}(W)$ denote the Hecke algebra associated with $(W,S)$. This is a $\mathcal{L}$-algebra spanned by a standard basis $\{h_\omega\}_{\omega\in W}$\footnote{A common distinct normalization appearing in literature would be $T_\omega: = v^{-\ell(\omega)}h_\omega$.}
, subject to the multiplicative relations
\[
h_{\omega_1}h_{\omega_2} = h_{\omega_1\omega_2},\quad\mbox{for all }\omega_1,\omega_2\in W,\mbox{ such that }\ell(\omega_1)+\ell(\omega_2) = \ell(\omega_1\omega_2)\;,
\]
\[
(h_s+v)(h_s-v^{-1})=0,\quad\mbox{for all }s\in S\;.
\]

One encompassing reference for the basics of this theory is \cite{S97}.

	
It will be convenient to formulate the basic Kazhdan-Lusztig theory through basis dualities.

Let us denote the dual $\mathcal{L}$-module
\[
\mathcal{H}(W)^* = \mathrm{Hom}_{\mathcal{L}}(\mathcal{H}(W),\mathcal{L})
\]
with the natural pairing $\langle\,,\,\rangle :\mathcal{H}(W)^*\times \mathcal{H}(W) \to \mathcal{L}$.

The dual standard basis $\{h^\omega\}_{\omega\in W}$ for $\mathcal{H}(W)^\ast$ would then be given by
\[
\langle h^{\omega'}, h_\omega\rangle =\delta_{w,w'},\forall \omega,\omega'\in W\;.
\]
	
We let $d$ also stand for the ring involution of $\mathcal{H}(W)$ given by
	\[
d(v):= v^{-1}, d(h_\omega):=(h_{\omega^{-1}})^{-1}\; \forall \omega\in W\;,
\]
which extends the previously defined involution on $\mathcal{L}$.

	By \cite{KL79}, there exists a unique basis $\{c_\omega\}_{\omega\in W}$ characterized by the properties
\[
d(c_\omega)=c_\omega,\quad \langle h^{\omega}, c_\omega\rangle = 1,\quad \langle h^\sigma, c_\omega\rangle \in v\mathbb{Z}[v] \subset \mathcal{L}\;, \forall\, \sigma,\omega\in  W\;.
\]
We define the (unnormalized) Kazhdan-Lusztig polynomials (associated with $(W,S)$) as
\[
\check{P}_{\sigma,\omega} = \langle h^\sigma, c_\omega\rangle \in \mathcal{L} \;,
\]
for all $\sigma,\omega\in W$.

The common normalization, originating in celebrated geometric interpretations, defines the Kazhdan-Lusztig polynomials as elements $P_{\sigma,\omega}\in \mathbb{Z}[q]\subset \mathcal{L}$ that satisfy
\[
P_{\sigma,\omega} = v^{\ell(\sigma)-\ell(\omega)} \check{P}_{\sigma,\omega}\;.
\]

It will be useful to switch in between the normalized and unnormalized notation according to computational needs.

We recall that $\check{P}_{\sigma,\omega} = P_{\sigma,\omega} = 0$, unless $\sigma\leq \omega$.

	%

\subsection{$R$-polynomials and $q$-derived Kazhdan-Lusztig polynomials}

The information encoded in Kazhdan-Lusztig polynomials can often be accessed through a similar notion of $R$-polynomials, whose relevance we now recall.

	Let $d^\ast$ denote the $\mathcal{L}$-linear involution on $\mathcal{H}(W)^\ast$ given by
\[
d^\ast(\phi):=d\circ\phi\circ d,\;\forall \phi\in \mathcal{H}(W)^*\;.
\]

We define the \textit{$R$-polynomials}
\[
\check{R}_{\sigma,\omega} = \langle d^\ast(h^\sigma), h_\omega\rangle\in \mathcal{L}\;,
\]
for all $\sigma,\omega\in W$. Again, it is a common convention to normalize the notion as
\[
R_{\sigma,\omega} = v^{\ell(\sigma)-\ell(\omega)}\langle d^\ast(h^\sigma), h_\omega\rangle\in \mathcal{L}\;.
\]
Same as with Kazhdan-Lusztig polynomials, it is easy to establish that $R_{\sigma,\omega} = \check{R}_{\sigma,\omega} = 0$, unless $\sigma\leq \omega$.

It is evident that the basis $\{d^\ast(h^\omega)\}_{\omega\in W}$ for $\mathcal{H}(W)^\ast$ is dual to the basis $\{d(h_\omega)\}_{\omega\in W}$. Thus, bearing in mind that the basis $\{c_\omega\}_{\omega\in W}$ is $d$-invariant, the identity
\begin{equation}\label{eq:0}
 d(\check{P}_{\sigma,\omega})=	\langle d^\ast(h^\sigma), c_\omega \rangle = \sum_{\kappa\in W} \langle d^\ast(h^\sigma), h_\kappa\rangle \langle h^\kappa, c_\omega \rangle = \sum_{\sigma \leq \kappa\leq \omega} \check{R}_{\sigma,\kappa} \check{P}_{\kappa, \omega} = v^{\ell(\omega)-\ell(\sigma)} \sum_{\sigma \leq \kappa\leq \omega} R_{\sigma,\kappa} P_{\kappa, \omega}
\end{equation}
is obtained.

	
	

It brings forth the notion of \textit{$q$-derived Kazhdan-Lusztig polynomials}, whose unnormalized variant we define as
\[
\check{P}^{\partial}_{\sigma,\omega} = \alpha^{-1} \langle d^\ast(h^\sigma) - h^{\sigma} , c_\omega \rangle = \frac{d(\check{P}_{\sigma,\omega}) - \check{P}_{\sigma,\omega}}{v^{-1}-v}\in \mathcal{L}\;,
\]
for $\sigma,\omega\in W$.

Note, that $\check{P}^\partial_{\sigma,\omega}$ is a well-defined $d$-invariant Laurent polynomial. For computational purposes, the fact that $\check{P}_{\sigma,\omega}$ belongs to the ring $v\mathbb{Z}[v]$, makes $d(\check{P}_{\sigma,\omega})$ belong to the ring $v^{-1}\mathbb{Z}[v^{-1}]$ and the value of $P_{\sigma,\omega}$ be directly extractable out of the value of $\check{P}^\partial_{\sigma,\omega}$.

The normalized variants of $q$-derived Kazhdan-Lusztig polynomials are expressed as the Laurent polynomials
\[
P^\partial_{\sigma,\omega} = v^{\ell(\sigma)-\ell(\omega)+1} \check{P}^\partial_{\sigma,\omega} = \frac{q^{\ell(\omega)-\ell(\sigma)} d(P_{\sigma,\omega})-P_{\sigma,\omega}}{q-1} \in \mathbb{Z}[q]\;.
\]
Since $R_{\sigma,\sigma} = 1$ for all $\sigma\in W$, the identity \eqref{eq:0} may now be read as
	\begin{equation}\label{eq:1}
		P^\partial_{\sigma,\omega} = \frac{1}{q-1} \sum_{\sigma < \kappa\leq \omega} R_{\sigma,\kappa} P_{\kappa, \omega}\;.
	\end{equation}

Indeed, the above form is often used as the standard method for recursive computations of Kazhdan-Lusztig polynomials.

%

	

\section{Positive decompositions}

\subsection{Dyer-Lehrer bases}
	Following Dyer-Lehrer \cite{DL90}, we define the elements
\[
f_{\omega,\tau} := h^{-1}_\tau h_{\tau\omega}\in \mathcal{H}(W)\;,
\]
for $\omega,\tau\in W$. For a fixed $\tau\in W$, $\{f_{\omega,\tau}\}_{\omega\in W}$ is clearly a basis for $\mathcal{H}(W)$. We write $\{f^{\omega,\tau}\}_{\omega\in W}$ for its dual basis in $\mathcal{H}(W)^\ast$.
	
Since for any $\omega\in W$, we have $\ell(\omega_0\omega)+\ell(\omega^{-1}) = \ell(\omega_0)$ and consequently $h_{\omega_0\omega} h_{\omega^{-1}} = h_{\omega_0}$, we deduce that
	\begin{equation}\label{eq:invd}
	f_{\omega,\omega_0} =h_{\omega_0}^{-1}h_{\omega_0\omega}= h_{\omega^{-1}}^{-1} = d(h_\omega)\;.
	\end{equation}
	
	Furthermore, for all $\omega,\tau\in W$, we have
	\[
	d(f_{\omega,\tau}) = d(h_{\tau})^{-1}d(h_{\tau\omega}) = f_{\tau,\omega_0}^{-1} f_{\tau\omega, \omega_0} =(h_{\omega_0}^{-1}h_{\omega_0\tau})^{-1}h_{\omega_0}^{-1}h_{\omega_0\tau\omega}= f_{\omega, \omega_0 \tau}\;.
	\]
	In particular,
\begin{equation}\label{eq:omega0}
d^\ast(f^{\omega,\tau}) = f^{\omega,\omega_0\tau}
\end{equation}
holds in $\mathcal{H}(W)^\ast$.

Much of the subsequent analysis of this subsection is implicitly contained in the work of Dyer in \cite{D93}. For clarity of discussion, we prefer to leave the arguments self-contained.

	\begin{lemma}\label{recursion}
		Let $\tau\in W$ be given and $s\in S$ be a simple reflection such that $\tau < s\tau$ holds. Let us write $t =\tau^{-1} s \tau\in W$. Then, for all $\omega\in W$, we have equalities
		\[
		f_{\omega, s\tau} = \left\{\begin{array}{ll} f_{\omega, \tau} & t\omega > \omega \\ f_{\omega,\tau} - \alpha f_{t\omega,\tau} & t\omega < \omega    \end{array}\right.\;,
		\]
in $\mathcal{H}(W)$, or dually in $\mathcal{H}(W)^\ast$,
		\[
		f^{\omega, s\tau} = \left\{\begin{array}{ll} f^{\omega, \tau} & t\omega > \omega \\ f^{\omega,\tau} + \alpha f^{t\omega,\tau} & t\omega < \omega    \end{array}\right.\;.
		\]
	\end{lemma}
	
	\begin{proof}
		Suppose that $t\omega > \omega$ holds. Since $\tau t > \tau$, we have $s\tau \omega = \tau t\omega > \tau\omega$ by a well-known property of Coxeter groups (\cite[Lemma 2.2.10]{BB05}). Hence,
		\[
		f_{\omega, s\tau} = (h_s h_\tau)^{-1} (h_s h_{\tau \omega}) = f_{\omega,\tau}\;.
		\]
		Similarly, when $t^2\omega > t\omega$, we have $\tau\omega = \tau t^2 \omega > \tau t \omega = s\tau \omega$. Thus, in the latter case, we obtain
		\[
		h_{\tau t \omega} = h^{-1}_s h_{\tau \omega} = (h_s-\alpha) h_{\tau\omega} = h_s (h_{\tau\omega} - \alpha h_{\tau t\omega})  \;,
		\]
		which implies $f_{\omega, s\tau} = (h_s h_\tau)^{-1} h_{\tau t \omega } = h_{\tau}^{-1} ( h_{\tau\omega} - \alpha h_{\tau t\omega})$.

The last identity follows, since $f_{t\omega,s\tau} = f_{\omega,\tau}$ as in the former case.
		
	\end{proof}

	Let $T\subset W$ denote the set of reflections in the Coxeter group, that is, $T = \bigcup_{\omega\in W} \omega S\omega^{-1}$.

For $\tau,\kappa\in W$, we say that a tuple $t_1,\ldots,t_k\in T$ is $(\tau,\kappa)$-\textit{convex}, if $\kappa = \tau t_1\ldots t_k$, and $\ell(\tau t_1\ldots t_i) = \ell(\tau) + i$, for all $1\leq i\leq k$.

It can be shown that if $t_1,\ldots, t_k\in T$ give a $(\tau,\kappa)$-convex sequence, then all $t_i$ are distinct.

A $(e, \omega_0)$-convex tuple amounts to a linear ordering of $T$. Those orderings are known as the reflection orderings, or the convex orderings of the positive roots in the root system associated with $W$, in cases where $W$ is a Weyl group.

A link of the above notion to the weak (left) Bruhat order $\leq_L$ on $W$ is naturally in place. We recall that $\tau \leq_L \kappa$ holds, whenever it is possible to write $\kappa = s_1\cdot\ldots\cdot s_k \tau$ with $s_1,\ldots, s_k\in S$ and $\ell(\kappa) = \ell(\tau)+k$.

It now easily follows that $\tau <_L \kappa$ holds, if and only if, a $(\tau,\kappa)$-convex tuple exists.
	
	\begin{proposition}\label{prop:Rformula}
		Let $t_1,\ldots,t_k\in T$ be a fixed $(\tau,\kappa)$-convex tuple.
		
		Then, for all distinct $\sigma,\omega\in W$, the formula
		\[
		\langle f^{\sigma,\tau} , f_{\omega,\kappa}\rangle = \sum_{\rho\in A(\sigma,\omega)} (-\alpha)^{n(\rho)}
		\]
		holds, where $A(\sigma,\omega)$ is the set of all sequences $\rho = (1\leq i_1 < \ldots < i_{n(\rho)}\leq k)$, for which
		\[
		\omega = t_{i_1}\cdot\ldots\cdot t_{i_{n(\rho)}} \sigma >  t_{i_2}\cdot\ldots\cdot t_{i_{n(\rho)}} \sigma > \ldots >  t_{i_{n(\rho)}} \sigma> \sigma
		\]
		is satisfied.

In addition, we have $\langle f^{\omega,\tau} , f_{\omega,\kappa} \rangle = 1$, for all $\omega\in W$.

	\end{proposition}
	\begin{proof}
		A recursive application of Lemma \ref{recursion} shows that
\[
f_{\omega,\kappa} = f_{\omega,\tau t_1\cdots t_k} = f_{\omega, \tau} + \sum_{\rho = (i_1,\ldots,i_{n(\rho)})\in A(\omega)} (-\alpha)^{n(\rho)} f_{t_{i_{n(\rho)}}\cdot\ldots\cdot t_{i_1}\omega,\tau}\;,
\]
where $A(\omega)$ stands for all sequences $\rho = (1\leq i_1 < \ldots < i_{n(\rho)}\leq k)$ with
	\[
		\omega > t_{i_1}\omega > t_{i_2}t_{i_1}\omega > \ldots >  t_{i_{n(\rho)}}\cdot\ldots \cdot t_{i_1} \omega\;.
		\]
In other words, $A(\omega) = \sqcup_{\sigma\in W} A(\sigma,\omega)$.

	\end{proof}

\subsection{Dyer filtrations}

For given elements $\tau, \sigma\in W$, according to \eqref{eq:omega0} we may write
\[
d^\ast(h^\sigma) - h^\sigma = (f^{\sigma,\tau} - h^{\sigma}) - d^\ast(f^{\sigma,\omega_0\tau} - h^{\sigma})\in \mathcal{H}(W)^\ast\;.
\]
It follows that for all $\tau,\sigma, \omega\in W$, we can decompose the $q$-derived Kazhdan-Lusztig polynomials in the form
\begin{equation}\label{eq:IQ}
  \check{P}^\partial_{\sigma,\omega} = \check{I}^\tau_{\sigma,\omega} + \check{Q}^\tau_{\sigma,\omega} = \check{I}^\tau_{\sigma,\omega}+ d(\check{I}^{\omega_0\tau}_{\sigma,\omega}) \;,
\end{equation}
where
\[
\check{I}^\tau_{\sigma,\omega}: = \alpha^{-1} \langle f^{\sigma,\tau} - h^{\sigma}, c_\omega\rangle \in \mathcal{L}\;,
\]
\[
\check{Q}^\tau_{\sigma,\omega}: = -\alpha^{-1} \langle d^\ast( f^{\sigma,\omega_0\tau} - h^{\sigma}), c_\omega \rangle  = d(\alpha^{-1} \langle f^{\sigma,\omega_0\tau} - h^{\sigma}, c_\omega \rangle )\in \mathcal{L}\;.
\]
The last equality exploits the fact that $c_\omega$ is $d$-invariant.


It follows from \eqref{eq:invd}	that $\check{I}^{\omega_0}_{\sigma,\omega} = \check{Q}^{e}_{\sigma,\omega} = \check{P}^{\omega_0}_{\sigma,\omega}$ and $\check{I}^{e}_{\sigma,\omega} = \check{Q}^{\omega_0}_{\sigma,\omega} = 0$.
Hence, the following proposition gives a sense in which the family of polynomials $\{\check{I}^\tau_{\sigma,\omega}\}_{\tau\in W}$ serves as a filtration of the $q$-derived Kazhdan-Lusztig polynomials.

	\begin{proposition}\label{prop:posIQ}
		For all $\sigma,\omega,\tau_1,\tau_2\in W$ with $\tau_1\leq_L\tau_2$, we have the monotonicity property
		\[
		\check{I}^{\tau_2}_{\sigma,\omega} - \check{I}^{\tau_1}_{\sigma,\omega}\in \mathcal{L}^+\;.
		\]
	\end{proposition}
	
	\begin{proof}
		It is enough to assume that $\tau_1 < \tau_2= s\tau_1$, for $s\in S$. Now, by \Cref{recursion}, $\check{I}^{\tau_2}_{\sigma,\omega} - \check{I}^{\tau_1}_{\sigma,\omega}$ equals either $0$ or $f^{\sigma',\tau_1}(c_{\omega})$ for $\sigma'=\tau_1^{-1}s\tau_1\sigma\in W$. Positivity now follows from \cite{DL90}.

	\end{proof}
	
	\begin{corollary}\label{cor:ppos}
		
		For all $\sigma,\omega,\tau\in W$, $\check{I}^{\tau}_{\sigma,\omega}$ and $\check{Q}^{\tau}_{\sigma,\omega}$ are in $\mathcal{L}^+$.
	\end{corollary}

\begin{proof}
Since $\check{I}^e_{\sigma,\omega} = 0$, the statement for $\check{I}^{\tau}_{\sigma,\omega}$ follows from \Cref{prop:posIQ}. The involution $d$ clearly preserves $\mathcal{L}^+$.

\end{proof}

	Note that in similarity with \eqref{eq:1} we can write
	\[
	\check{I}^{\tau}_{\sigma,\omega} = \alpha^{-1} \sum_{\sigma < \kappa\leq \omega} \langle f^{\sigma,\tau}, h_\kappa \rangle  \check{P}_{\kappa,\omega}\;,
	\]
	and
	\[
    \check{Q}^{\tau}_{\sigma,\omega} = -\alpha^{-1}\sum_{\sigma < \kappa\leq \omega} \langle f^{\sigma,\tau}, d(h_\kappa)\rangle  d(\check{P}_{\kappa,\omega})\;.
	\]
	Thus, we see that, for a fixed $\tau\in W$, a recursive formula for $P_{\sigma,\omega}$ may be obtained by computing the values of $\langle f^{\kappa_1,\tau}, h_{\kappa_2}\rangle$ and $\langle f^{\kappa_1,\tau}, d(h_{\kappa_2})\rangle$, for $\sigma\leq  \kappa_1< \kappa_2\leq \omega$.

	
We also introduce the normalized variants
\[
I^\tau_{\sigma,\omega} := v^{\ell(\sigma)-\ell(\omega)+1} \check{I}^\tau_{\sigma,\omega},\quad Q^\tau_{\sigma,\tau} := v^{\ell(\sigma)-\ell(\omega)+1} \check{Q}^\tau_{\sigma,\omega} = v^{2(\ell(\sigma)-\ell(\omega))+2} d(I^\tau_{\sigma,\omega} )\in\mathcal{L}\;,
\]
so that
\[
  P^\partial_{\sigma,\omega} = I^\tau_{\sigma,\omega} + Q^\tau_{\sigma,\omega}
\]
holds, for all $\sigma,\omega,\tau\in W$.

	\subsection{A parabolic recursion formula - I}

We would like to examine special cases of decompositions of the form of \eqref{eq:IQ}, which present an algorithmic descent into values of Kazhdan-Lusztig polynomials of a Coxeter group of a smaller rank, that is, a parabolic subgroup of $W$.



For a subset $J\subset S$ of simple reflections, we write $W_J< W$ for the (parabolic) subgroup generated by $J$.

Recall that each of cosets in $W_J \setminus W$ contains a unique element of minimal length. Let ${}^J W \subset W$ be the set of all such minimal length representatives of this coset space.

For $\omega \in W$, we may now write $\omega = \omega_J {}^J\omega$ for the unique elements $\omega_J\in W_J$ and ${}^J\omega \in {}^J W$.

Following Grojnowski-Haiman \cite{GH07}, we define the elements
\[
h_{\omega,J}: = c_{\omega_J} h_{{}^J\omega}\in \mathcal{H}(W)\;,
\]
that give the $J$-\textit{hybrid} basis $\{h_{\omega,J}\}_{\omega\in W}$ for $\mathcal{H}(W)$. We denote by $\{h^{\omega,J}\}_{\omega\in W}$ its dual basis in $\mathcal{H}(W)^\ast$.

It is a consequence of \cite{GH07} that
\[
\check{\gamma}^J_{\sigma,\omega}:=\langle h^{\sigma,J}, c_\omega \rangle\in \mathcal{L}^+\;,
\]
for all $\sigma,\omega\in W$ and $J\subset S$.



Let us examine the behavior of the $J$-hybrid basis relative to the family of Dyer-Lehrer bases, and, in particular, relative to the standard basis.

\begin{lemma}\label{lem:1}
	For all $\tau\in W_J$ and $\sigma,\omega\in W$, we have
	\[
	\langle f^{\sigma,\tau}, h_{\omega,J} \rangle = \left\{ \begin{array}{ll} \langle f^{\sigma_J, \tau}, c_{\omega_J} \rangle & ^J\sigma = {}^J\omega \\ 0 &  {}^J\sigma \neq {}^J\omega \end{array}\right. \;.
	\]
\end{lemma}
\begin{proof}
	We note that $f_{\zeta,\tau} = f_{\zeta_J, \tau}h_{{}^J\zeta}$ holds, for all $\zeta\in W$. Since $\ell(\tau\zeta)=\ell(\tau\zeta_J)+\ell(^J\zeta)$, the identity follows from multiplying the equation
	\[
	c_{\omega_J} = \sum_{\upsilon \in W_J} \langle f^{\upsilon,\tau} , c_{\omega_J} \rangle f_{\upsilon, \tau}
	\]
	by $h_{^J\omega}$ on the right side.
\end{proof}

The following proposition takes note that the polynomials $\check{\gamma}^J_{\sigma,\omega}$ are linearly computable out of values of corresponding Kazhdan-Lusztig polynomials for the groups $W$ and $W_J$.

In order to formulate it in a normalized fashion, let us introduce the normalized variant of Grojnowski-Haiman polynomials: We define $\gamma^J_{\sigma,\omega}:= v^{\ell(\sigma)-\ell(\omega)} \check{\gamma}^J_{\sigma,\omega}$, for all $\sigma,\omega\in W$ and $J\subset S$.

\begin{proposition}\label{prop:lin}
	For all $\sigma,\omega\in W$, we have
	\[
	\sum_{\kappa\in W_J} P_{\sigma_J, \kappa} \gamma^J_{\kappa\cdot {}^J\sigma, \omega} = P_{\sigma, \omega}\;.
	\]
	In particular, for a fixed $\upsilon\in {}^J W$, we may write in matrix notation
	\[
	\left( P_{\zeta, \kappa}\right)_{\zeta,\kappa\in W_J}\cdot (\gamma^J_{\kappa \upsilon, \omega})_{\kappa\in W_J} = (P_{\zeta\upsilon, \omega})_{\zeta\in W_J}\;.
	\]
	
\end{proposition}

\begin{proof}
	By Lemma \ref{lem:1} (with $\tau=e$),
	\[
	\langle h^{\sigma}, c_\omega \rangle = \sum_{\kappa'\in W} \langle h^{\sigma}, h_{\kappa',J}\rangle \langle h^{\kappa',J}, c_\omega\rangle = \sum_{\kappa\in W_J} \langle h^{\sigma_J}, c_\kappa\rangle \langle h^{\kappa \cdot {}^J\sigma, J}, c_\omega\rangle\;.
	\]
Normalization yields the results, when noting that for $\kappa\in W_J$,
\begin{equation*}
\begin{aligned}
\ell(\sigma) - \ell(\omega) &= \ell(\sigma_J) + \ell({}^J\sigma) - \ell(\omega)\\
&= \ell(\sigma_J) -  \ell(\kappa) +\ell(\kappa ) + \ell({}^J\sigma) - \ell(\omega)\\
&= (\ell(\sigma_J) -  \ell(\kappa) ) +( \ell(\kappa\cdot {}^J\sigma ) - \ell(\omega))\;.
\end{aligned}
\end{equation*}

\end{proof}

\begin{proposition}\label{prop:inductive}
	For $\tau\in W_J$ and $\sigma,\omega \in W$, we have the decompositions
	\[
	\check{I}^\tau_{\sigma,\omega} = \sum_{\kappa\in W_J\;:\; \sigma < \kappa\cdot {}^J\sigma \leq \omega } \check{I}^\tau_{\sigma_J,\kappa} \check{\gamma}^{J}_{\kappa\cdot {}^J\sigma,\omega}\;,
	\]
and
\[
I^\tau_{\sigma,\omega} = \sum_{\kappa\in W_J\;:\; \sigma < \kappa\cdot {}^J\sigma \leq \omega } I^\tau_{\sigma_J,\kappa} \gamma^{J}_{\kappa\cdot {}^J\sigma,\omega}\;,
\]
	where each of the summands is in $\mathcal{L}^+$.
\end{proposition}
\begin{proof}
	In similarity with the argument in the proof of \Cref{prop:lin}, we may write
\[
\check{I}^\tau_{\sigma,\omega} = \alpha^{-1}\sum_{\kappa\in W_J} \langle f^{\sigma_J,\tau}-  h^{\sigma_J}, c_{\kappa} \rangle \langle h^{\kappa\cdot{}^J\sigma,J}, c_\omega \rangle\;,
\]
and repeat the same proof's argument for the normalized variant.
\end{proof}




Let us write $w_0^J\in W_J$ for the longest element in the parabolic subgroup.

For any $\sigma,\omega\in W$ and $J\subset S$, we denote the elements
\[
I^J_{\sigma,\omega} := I^{\omega_{0}^J}_{\sigma,\omega},\quad Q^J_{\sigma,\omega} := Q^{\omega_0^J}_{\sigma,\omega},\quad \check{I}^J_{\sigma,\omega} := \check{I}^{\omega_{0}^J}_{\sigma,\omega},\quad \check{Q}^J_{\sigma,\omega} := \check{Q}^{\omega_0^J}_{\sigma,\omega}\;\in \mathcal{L}^+\;.
\]

\begin{corollary}\label{cor:indu}
	For $\sigma,\omega\in W$ and $J\subset S$, we have
	\[
	I^J_{\sigma,\omega} = \sum_{\kappa\in W_J\;:\; \sigma < \kappa\cdot {}^J\sigma \leq \omega } P^{\partial}_{\sigma_J,\kappa} \gamma^{J}_{\kappa\cdot {}^J\sigma,\omega}\;.
	\]
	
\end{corollary}

\begin{proof}
Treating $I^{\omega_0^J}_{\sigma_J,\kappa}$, for $\sigma \in W$ and $\kappa\in W_J$, as polynomials that are defined in terms of the algebra $\mathcal{H}(W_J)$, we see the equality $I^{\omega_0^J}_{\sigma_J,\kappa}= P^\partial_{\sigma_J,\kappa}$ coming from \eqref{eq:invd}. Thus, the statement follows from \Cref{prop:inductive}.

\end{proof}


Having expressed $I^J_{\sigma,\omega}$ in recursive terms that depend on the parabolic subgroup $W_J$, we would like to describe the remainder $Q^J_{\sigma,\omega} = P^\partial_{\sigma,\omega} -I^J_{\sigma,\omega}$ through a recursion of a similar nature to the familiar identity \eqref{eq:1}.

For $\sigma,\omega\in W$ and $J\subset S$, let us define the \textit{$J$-relative $R$-polynomials} as
\[
\check{R}_{\sigma,\omega,J} :=  \langle f^{\sigma, \omega_0\omega_{0}^J}, h_{\omega}\rangle =  d(\langle  f^{\sigma, \omega_{0}^J}, d(h_{\omega})\rangle )\in \mathcal{L}\;,
\]
with the last equality being a consequence of \eqref{eq:invd}, and their normalized variants
\[
R_{\sigma,\omega,J} := v^{\ell(\sigma)-\ell(\omega)}\check{R}_{\sigma,\omega,J}\;.
\]

Note, that $R_{\sigma,\omega,\emptyset}=R_{\sigma,\omega}, R_{\sigma,\omega,S}=\delta_{\sigma,\omega}$.


The analog of \eqref{eq:1} is now easily seen to hold.

\begin{proposition}\label{prop:qR}
	For all $\sigma,\omega\in W$ and $J\subset S$, the identity
\[
Q^J_{\sigma,\omega} =  \frac{q^{\ell(\omega)-\ell(\sigma)}}{1-q} d\left( \sum_{\sigma < \kappa\leq \omega} R_{\sigma,\kappa,J} P_{\kappa, \omega}\right)
\]
holds.
\end{proposition}
\begin{proof}
Using the facts $\check{Q}^J_{\sigma,\omega}= d(\check{I}^{\omega_0\omega_0^J}_{\sigma,\omega})$ and $R_{\sigma,\sigma,J}=1$ (\Cref{prop:Rformula}), this is the normalized form of the equality
\[
\langle f^{\sigma,\omega_0\omega_0^J} - h^{\sigma}, c_\omega \rangle = \sum_{\kappa\in W} \langle f^{\sigma,\omega_0\omega_0^J} , h_\kappa\rangle \langle h^\kappa, c_\omega \rangle -  \langle h^{\sigma}, c_\omega \rangle = \sum_{\sigma < \kappa\leq \omega} \langle f^{\sigma,\omega_0\omega_0^J} , h_\kappa\rangle \langle h^\kappa, c_\omega \rangle\;.
\]

\end{proof}


Our main \Cref{thm:A} now readily follows from \Cref{cor:indu} and \Cref{prop:qR}.

\begin{remark}\label{rem:1}
An evident corollary of \Cref{thm:A} (and Lemma \ref{prop:lin}) is an algorithm that computes the value of $P_{\sigma,\omega}$, for given $\sigma,\omega\in W$, out of the values of the polynomials $R_{\sigma,\zeta,J}$ and $P_{\zeta,\omega}$ for $\sigma < \zeta \leq \omega$ and $P_{\sigma_J, \kappa}$ for $\sigma_J < \kappa \leq c$, where $c\in W_J$ is the maximal element with $c \cdot{}^J\sigma \leq \omega$.
\end{remark}

\subsection{ A parabolic recursion formula - II}

Let us revisit the familiar theme of Kazhdan-Lusztig theory that follows from the identity \eqref{eq:1}, where a computation of the polynomials $P_{\sigma,\omega}$ is recursively substituted by computations of the polynomials $R_{\sigma,\omega}$. While a geometric handle on the situation, with its strong positivity benefits, is often lost in the process, simpler algorithmic formulas such as in \Cref{prop:Rformula} may present an advantage.

Indeed, the decomposition of \Cref{thm:A} in terms of $J$-relative $R$-polynomials can also be stated without an explicit use of the Kazhdan-Lusztig basis $\{c_\omega\}_{\omega\in W}$.

To that aim let us first observe a natural extension of the phenomenon in \Cref{lem:1} that does not involve the $J$-hybrid basis.

\begin{lemma}\label{lem:2}
	For all $\sigma,\omega\in W$, $J\subset S$ and $\tau,\kappa\in W_J$, we have
	\[
	\langle f^{\sigma,\tau}, f_{\omega,\kappa} \rangle = \left\{ \begin{array}{ll} \langle f^{\sigma_J, \tau}, f_{\omega_J,\kappa} \rangle & ^J\sigma = {}^J\omega \\ 0 &  {}^J\sigma \neq {}^J\omega \end{array}\right. \;.
	\]

As a consequence,
\[
\langle f^{\sigma,\omega_0}, f_{\omega,\omega_0\omega_0^J} \rangle = \left\{ \begin{array}{ll} \check{R}_{\sigma_J,\omega_J} & ^J\sigma = {}^J\omega \\ 0 &  {}^J\sigma \neq {}^J\omega \end{array}\right.
\]
holds, and $\langle f^{\sigma,\omega_0\omega_0^J}, f_{\omega,\omega_0} \rangle = d(\langle f^{\sigma,\omega_0}, f_{\omega,\omega_0\omega_0^J} \rangle)$.
\end{lemma}
\begin{proof}
The first statement is proved as in \Cref{lem:1} mutatis mutandis.

It now follows that
\begin{equation*}
\begin{aligned}
\langle f^{\sigma,\omega_0}, f_{\omega,\omega_0\omega_0^J} \rangle &= d\left(\langle d^\ast(f^{\sigma,\omega_0}), d(f_{\omega,\omega_0\omega_0^J}) \rangle\right) \\
&= d\left(\langle h^{\sigma}, f_{\omega,\omega_0^J} \rangle\right)\\
&=\left\{ \begin{array}{ll}  d\left(\langle h^{\sigma_J}, f_{\omega_J,\omega_0^J} \rangle\right) & ^J\sigma = {}^J\omega \\ 0 &  {}^J\sigma \neq {}^J\omega \end{array}\right.\;,
\end{aligned}
\end{equation*}
which gives the expression for $\check{R}_{\sigma_J,\omega_J}$ when bearing in mind that $h_{\omega_J}$ is sent to $ f_{\omega_J,\omega_0^J}$ by the $d$-involution for the algebra $\mathcal{H}(W_J)$.

The final equality follows from a similar consideration.
\end{proof}

\begin{proposition}\label{prop:recJR}
  For all $\sigma,\omega\in W$ and $J\subset S$, the formulas
  \[
  \check{R}_{\sigma,\omega} = \sum_{\kappa\in W_J\;:\; \sigma \leq \kappa\cdot{}^J\sigma \leq \omega } \check{R}_{\sigma_J, \kappa} \check{R}_{\kappa \cdot{}^J\sigma, \omega, J},\quad R_{\sigma,\omega} = \sum_{\kappa\in W_J\;:\; \sigma \leq \kappa\cdot{}^J\sigma \leq \omega } R_{\sigma_J, \kappa} R_{\kappa\cdot {}^J\sigma, \omega, J}
  \]
  hold.

  The inverse relation
  \[
  R_{\sigma,\omega,J} = \sum_{\kappa\in W_J\;:\; \sigma \leq \kappa\cdot{}^J\sigma \leq \omega } d(R_{\sigma_J, \kappa}) R_{\kappa \cdot{}^J\sigma, \omega}
  \]
  is valid as well.
\end{proposition}

\begin{proof}
By \Cref{lem:2} the identities amounts to
\[
\langle f^{\sigma,\omega_0}, h_{\omega} \rangle  = \sum_{\kappa'\in W} \langle f^{\sigma,\omega_0}, f_{\kappa', \omega_0\omega_0^J} \rangle \langle f^{\kappa',\omega_0\omega_0^J}, h_{\omega} \rangle\;,
\]
with the normalization constants cancelling out as in the proof of \Cref{prop:lin}.

For the inverse relation we change the order of basis change to reach
\[
\langle f^{\sigma,\omega_0\omega_0^J}, h_{\omega} \rangle  = \sum_{\kappa'\in W} \langle f^{\sigma,\omega_0\omega_0^J}, f_{\kappa',\omega_0} \rangle \langle f^{\kappa',\omega_0}, h_{\omega} \rangle\;,
\]
and apply the last equality in the statement of \Cref{lem:2}.

\end{proof}

Streamlining the above identity through a nested chain of parabolic subgroups serves an argument that $J$-relative $R$-polynomials may be viewed as computational building blocks for Kazhdan-Lusztig polynomials. Let us make this view explicit.

For a given chain $\emptyset = J_0 \subsetneq J_1 \subsetneq \ldots \subsetneq J_r = S$ and $\sigma\in W$, let us write the unique decomposition $\sigma = \sigma_1 \sigma_2\cdots \sigma_r$, so that $\sigma_i\in {}^{J_{i-1}, J_i} W$, where ${}^{J_{i-1}, J_i} W \subset W_{J_{i}}$ denotes the set of minimal length representatives of the coset space $W_{J_{i-1}} \setminus W_{J_i}$.

\begin{corollary}
For all $\sigma,\omega\in W$ with $\sigma$ decomposed as above, we have
\[
  R_{\sigma,\omega} = \sum_{(\kappa_1,\ldots,\kappa_{r-1})}  R_{\sigma_1, \kappa_1}   \left( \prod_{j=1}^{r-2}  R_{\kappa_{j} \sigma_{j+1}, \kappa_{j+1}, J_{j}} \right) R_{\kappa_{r-1} \sigma_{r}, \omega, J_{r-1}}\;,
\]
where the summation is over all tuples with $\kappa_i\in W_{J_i}$, $i=1,\ldots,r-1$.

\end{corollary}

As a follow-up of \Cref{rem:1}, the above corollary adds transparency to the approach that Kazhdan-Lusztig polynomials $P_{\sigma,\omega}$ (or their analogs $R_{\sigma,\omega}$) may be produced recursively out of the values of the polynomials $R_{\kappa^1_{J_i}, \kappa^2_{J_i}, J_{i-1}}$, for $\kappa^1,\kappa^2\in W$ such that $\sigma \leq \kappa^1  < \kappa^2  \leq \omega $ and $1\leq i \leq r$.

\section{The case of $S_n$}\label{sect:Sn}

In this section we fix the prototypical Coxeter system $(W, S)$, with $W=S_n$ being viewed as the group of permutations on $n$ indices $\{1,\ldots,n\}$, and its Coxeter generators are the simple transpositions
\[
S = \{s_i = (i,\, i+1)\}_{i=1}^{n-1}\subset S_n\;.
\]

Here, we adopt the common notation of $(i,j)\in S_n$ denoting the permutation that transposes given indices $1\leq i\neq j\leq n$ and fixes all other indices.

We also fix $J = \{s_1,\ldots, s_{n-2}\}\subset S$, which gives the parabolic subgroup $W_J <S_n$ of permutations that have the index $n$ as a fixed point. Naturally, the group $W_J$ is identified with the group $S_{n-1}$ of permutations on $\{1,\ldots,n-1\}$.

Our aim is to show that the $J$-relative $R$-polynomials in this case have an appealing description with implications towards the combinatorial invariance conjecture for Kazhdan-Lusztig polynomials.


Let us first introduce some further notation endemic to the case of permutation groups.

We naturally treat elements in $S_n$ as functions $\{1,\ldots,n\}\to \{1,\ldots,n\}$.

For a subset of indices $A = \{i_1 < \ldots < i_k\}\subset \{1,\ldots,n\}$ and a permutation $\sigma\in S_n$, the restriction $\sigma|_A \in S_k$ is defined to be the permutation satisfying $\sigma|_A(t) < \sigma|_A(s)$, if and only if, $\sigma(i_t) <\sigma(i_s)$ holds.

The following proposition is a well-known fact that was referred to as the \textit{subword condition} in \cite[Proposition 5.6]{BBDVW21}.

\begin{proposition}\label{prop:res-isom}
  For a choice of an index set $A\subset \{1,\ldots, n\}$ and of a permutation $\tau\in S_{|A|}$, let us denote
\[
\mathcal{I}_{A,\tau} = \{\sigma \in S_n \;:\; \sigma|_A = \tau\}\subset S_n\;.
\]

Then, the bijection $\sigma \mapsto \sigma|_{\{1,\ldots,n\}\setminus A}$ from $\mathcal{I}_{A,\tau}$ to $S_{n-|A|}$ is an isomorphism of partially ordered sets, relative to Bruhat orders on $S_n$ and $S_{n-|A|}$.

\end{proposition}

For flexibility, let us denote $\omega_{0,k}\in S_n$ the permutation taking $i$ to $k+1-i$, for $1\leq i\leq k$ and leaving all indices $k<i$ fixed. In particular, $\omega_{0,n} = \omega_0 \in S_n$ is the longest permutation, while $\omega_{0,n-1}=\omega_0^J\in W_J$ is the longest element in the parabolic subgroup.

\subsection{Computation of $R_{\sigma,\omega,J}$}

For $1\leq i < n$, we write $t_i = (i, n) \in T$. Note, that $t_{n-1} ,t_{n-2},\ldots, t_1$ becomes a $(\omega_0^J, \omega_0)$-convex tuple that we now fix.

\begin{proposition}\label{prop:cond}
Let $\sigma \in S_n$ be a permutation, and $A = \{i_1< \ldots < i_k\}\subset \{1,\ldots, n-1\}$ a set of indices. The following are equivalent conditions:
\begin{enumerate}
  \item\label{it-1} The inequalities
  \[
   \sigma < t_{i_1}\sigma < t_{i_2} t_{i_1} \sigma< \ldots  < t_{i_k} t_{i_{k-1}}\cdot\ldots\cdot t_{i_1} \sigma
  \]
  hold in the Bruhat order of $S_n$.
  \item\label{it-2} The equality $\omega_{0,k} = (\sigma^{-1})|_{A\cup \{n\}}$ holds, as permutations in $S_{k+1}$.
  \item\label{it-3} For each $i\in A$, $\sigma < t_i\sigma$ holds, and the permutations $\{t_i\sigma\}_{i\in A}$ are pairwise incomparable in the Bruhat order of $S_n$.
\end{enumerate}

\end{proposition}

\begin{proof}

When assuming condition \eqref{it-2}, the statements of \eqref{it-1} and \eqref{it-3} follow easily from Proposition \ref{prop:res-isom}.

Now, let us assume condition \eqref{it-3}. For $i,j\in A$ with $i<j$, let us consider the permutation $\tau = (\sigma^{-1})|_{\{i,j,n\}}\in S_3$. It follows that $\tau < \tau \cdot (1,3)$ and $\tau < \tau \cdot (2,3)$ hold, and that the pair $\tau \cdot(1,3) , \tau\cdot(2,3)$ is incomparable in the Bruhat order of $S_3$. Examining a small number of possibilities, we arrive to the conclusion that $\tau = (213)$. In other words, $\sigma^{-1}(j) < \sigma^{-1}(i) < \sigma^{-1}(n)$.
Observing all possible pairs in $A$, we obtain that $\sigma^{-1}(i_k) < \ldots < \sigma^{-1}(i_1) < \sigma^{-1}(n)$ holds, which amounts to condition \eqref{it-2}.

We are left with proving that \eqref{it-1} implies \eqref{it-2}. We show it by induction on $k$, the size of $A$.

By the induction hypothesis, we know that $(\sigma^{-1})|_{\{i_1,\ldots,i_{k-1},n\} }= \omega_{0,k-1}$. It remains to show that $\sigma^{-1}(i_k) < \sigma^{-1}(i_{k-1})$.

Indeed, noting that $\omega: = t_{i_{k-1}}\cdot\ldots\cdot t_{i_1} \sigma$ must satisfy $\omega^{-1}(n) = \sigma^{-1}(i_{k-1})$ and that $\omega < t_{i_k}\omega$ is assumed, the needed inequality follows.

\end{proof}

For a permutation $\sigma\in S_n$, let us write $\mathcal{S}(\sigma)$ for the collection of subsets of $\{1,\ldots,n-1\}$ that satisfy the equivalent conditions of Proposition \ref{prop:cond} relative to $\sigma$.

For $A =\{i_1< \ldots < i_k\}\in \mathcal{S}(\sigma)$, we write
\[
\sigma^A = t_{i_k} t_{i_{k-1}}\cdot\ldots\cdot t_{i_1} \sigma = (i_1, \ldots, i_k, n)\sigma \in S_n\;,
\]
where the last expression employs the cycle notation for permutations.

Taking the convention that $\emptyset\in \mathcal{S}(\sigma)$, we set $\sigma^{\emptyset} =\sigma$.

\begin{proposition}\label{prop:Runnorm}
  For all $\sigma,\omega\in S_n$, the formula
  \[
  \check{R}_{\sigma,\omega,J}  =  \left\{\begin{array}{ll} \alpha^{|A|}  & \mbox{if exists }A\in \mathcal{S}(\sigma),\mbox{ such that }\omega = \sigma^A \;.\\ 0 & \mbox{otherwise} \end{array}\right.
  \]
  holds.
\end{proposition}

\begin{proof}
Since $\check{R}_{\sigma,\omega,J} =  d( \langle f^{\sigma,\omega_0^J}, f_{\omega,\omega_0}\rangle )$, $d(\alpha) = -\alpha$ and $t_{n-1},\ldots, t_1$ is a $(\omega_0^J,\omega_0)$-convex tuple, the formula follows directly from Proposition \ref{prop:Rformula} and the definition of $\mathcal{S}(\sigma)$.

\end{proof}

Curiously, the normalized polynomials $R_{\sigma,\omega,J}$ possess their own natural description that avoids the need for a direct computation of the normalization constants $\ell(\sigma)-\ell(\sigma^A)$, for a given $A\in \mathcal{S}(\sigma)$.

To that aim we need a finer analysis of the combinatorial situation. For a fixed $\sigma\in S_n$, let us write $\mathcal{D}(\sigma)\subset \{1,\ldots, n-1\}$ for the set of indices $i$, that satisfy $\sigma < t_i\sigma$ (or, equivalently, $\sigma^{-1}(i) < \sigma^{-1}(n)$).

Let $\mathcal{P}(\sigma)$ denote the partially ordered set of all $2^{|\mathcal{D}(\sigma)|}$ subsets of $\mathcal{D}(\sigma)$.

We define a pair of inclusion preserving maps
\[
A  \mapsto  A^\sigma\;,\quad B  \mapsto  B_\sigma\;,
\]
from $\mathcal{P}(\sigma)$ to itself, as follows.

For $A\in \mathcal{P}(\sigma)$, we set
\[
A\subset A^\sigma = \{1\leq j< n \;:\; \exists i\in A,\; i\leq j,\; \sigma^{-1}(i)\leq \sigma^{-1}(j)\leq \sigma^{-1}(n)\}\;.
\]

Given $B\in \mathcal{P}(\sigma)$, we define $B_\sigma \subset B$ using (a slightly twisted version of) the algorithm described in \cite[Section 5.1]{BBDVW21}.

The index $i_1\in B$ is set to be the minimal index with $\sigma^{-1}(i_1) < \sigma^{-1}(n)$. Subsequently, $i_2\in B$ is set to be the minimal index larger than $i_1$, for which $\sigma^{-1}(i_2)< \sigma^{-1}(i_1)$. Proceeding inductively, we obtain $B_\sigma := \{i_1<\ldots< i_k\}$.

The following lemma is straightforward to verify.
\begin{lemma}
For all $A\in \mathcal{P}(\sigma)$, the inclusions $A\subset (A_\sigma)^\sigma$, $(A^\sigma)_\sigma\subset A$ and the equalities $(A_\sigma)_\sigma = A_\sigma$, $(A^\sigma)^\sigma = A^\sigma$ hold.

In particular, the pair of maps $A\mapsto A^\sigma$ and $B\mapsto B_\sigma$ form a Galois connection between the poset $\mathcal{P}(\sigma)$ to itself.

The image of the map $B\mapsto B_\sigma$ is precisely $\mathcal{S}(\sigma)$.
\end{lemma}

A description of the fibers of $B\mapsto B_\sigma$ now conveniently follows.

\begin{corollary}\label{lem:fiber}
For all $A\in \mathcal{S}(\sigma)$, we have
\[
\{B\in \mathcal{P}(\sigma)\;:\; B_\sigma = A\} = \{B\in \mathcal{P}(\sigma)\;:\; A\subset B \subset  A^\sigma\}\;.
\]
\end{corollary}

On the other hand, the set $A^\sigma$, for $A\in \mathcal{S}(\sigma)$, naturally appears in our computations of interest.

\begin{lemma}\label{lem:length}
For any $A\in \mathcal{S}(\sigma)$, the length identity
\[
\ell(\sigma^A) - \ell(\sigma) =  2|A^\sigma| - |A|
\]
holds.
\end{lemma}

\begin{proof}
Let us fix $A = \{i_1< \ldots< i_k\}\in \mathcal{S}(\sigma)$ and write $\omega = \sigma^A\in S_k$. We also write $A_1 = \{i_1< \ldots < i_{t-1}\}\in \mathcal{S}(\sigma)$ and $\omega_1 =  \sigma^{A_1} = t_{i_k} \omega $.

Let us consider the set of indices
\[
C = \{i_{k}< i < n\;:\; \sigma^{-1}(i_k)<\sigma^{-1}(i) < \sigma^{-1}(i_{k-1})\}\;.
\]
Here, $i_0$ is taken as $n$, if $k=1$.

A simple study of permutations shows that $\ell(t_{i_k}\omega_1) - \ell(\omega_1) = 2 |C| +1$. On the other hand, we see that $A^\sigma = A_1^\sigma \dot{\cup} C \dot{\cup} \{i_k\}$. Thus, we may write
\[
\ell(\omega ) - \ell(\omega_1) = 2(|A^\sigma| - |A_1^\sigma|)-1\;.
\]
Reasoning by induction on the parameter $|A|$, the statement follows.
\end{proof}

Given $\sigma \in S_n$ and $B\in \mathcal{P}(\sigma)$, let us shortcut notation to $\sigma^B := \sigma^{B_\sigma}\in S_n$.

\begin{proposition}\label{prop:normRj}
  For all $\sigma,\omega\in S_n$, the formula
  \[
  R_{\sigma,\omega,J}  = \sum\limits_{B\in \mathcal{P}(\sigma)\;:\; \sigma^B = \omega} (q-1)^{|B|}
  \]
  holds.
\end{proposition}
\begin{proof}
By \Cref{lem:fiber}, for $A\in \mathcal{S}(\sigma)$,
\begin{equation*}
\begin{aligned}
	\sum\limits_{B\in \mathcal{P}(\sigma)\;:\; B_\sigma = A} (q-1)^{|B|}
	&=\sum\limits_{D\subset A^\sigma\setminus A} (q-1)^{|A|+|D|} \\
	&=(q-1)^{|A|} \sum_{r=0}^{|A^\sigma|- |A|} {|A^\sigma|- |A|  \choose r}  (q-1)^{r} \\
	&= (q-1)^{|A|} q^{|A^\sigma|- |A|}\;.
	\end{aligned}
\end{equation*}

Now, by \Cref{lem:length}, $(q-1)^{|A|} q^{|A^\sigma|- |A|} = v^{\ell(\sigma)-\ell(\sigma^A)}v^{|A|} (v^{-2}-1)^{|A|} =v^{\ell(\sigma)-\ell(\sigma^A)}\alpha^{|A|} $. The result now follows from the unnormalized formula in \Cref{prop:Runnorm}.

\end{proof}

\subsection{Hypercube decomposition}

\begin{lemma}\label{lem:supr}
For $\sigma\in S_n$ and non-empty $B\in \mathcal{P}(\sigma)$, $\sigma^B$ is the supremum (join) of the set $\{t_i \sigma\}_{i\in B}$ in the Bruhat order.
\end{lemma}
\begin{proof}
Let us first see that the supremum of the set $\{t_i \sigma\}_{i\in B}$ is achieved as the supremum of the possibly smaller set $\{t_i \sigma\}_{i\in B_\sigma}$.

Indeed, given $j\in B\setminus B_\sigma$, there must be $j> i\in B$ with $\sigma^{-1}(i) < \sigma^{-1}(j)< \sigma^{-1}(n)$. In other words, the restricted permutation $(\sigma^{-1})|_{\{i,j,n\}}$ is trivial in $S_3$. Since $(2,3)< (1,3)$, it follows that $t_j\sigma < t_i \sigma$.

Hence, we can assume that $B\in \mathcal{S}(\sigma)$, and by \Cref{prop:cond} we have $(\sigma^B)^{-1}|_{B\cup \{n\}} = \omega_{0,k+1}$.

It is easily verified that the Bruhat supremum of the set of permutations
\[
\{\omega_{0,k}(1,k+1),\omega_{0,k}(2,k+1)\ldots, \omega_{0,k}(k,k+1)\}
\]
in $S_{k+1}$ is given by the longest element $\omega_{0,k+1}$.

The statement follows from Proposition \ref{prop:res-isom}.

\end{proof}

The combination of \Cref{prop:normRj} and \Cref{lem:supr} now gives the gist of \Cref{thm:B}. Furthermore, if we substitute the hypercube formula for $J$-relative $R$-polynomials into the general expression from \Cref{prop:qR}, we obtain an algebraic proof and interpretation of the hypercube decomposition presented in \cite{BBDVW21}. The following expression that we record separately allows us to complete the proof of \Cref{thm:B}.

\begin{theorem}\label{thm:hyper}
  For all $\sigma, \omega\in S_n$ and $J = \{s_1,\ldots, s_{n-2}\}$, the equality
  \[
  Q^J_{\sigma,\omega} =  \frac{q^{\ell(\omega)-\ell(\sigma)}}{1-q}  \sum\limits_{\emptyset\neq B\in \mathcal{P}(\sigma)} (q^{-1}-1)^{|B|} d(P_{\sigma^B, \omega})\;,
    \]
holds, where $\sigma^B\in S_n$ may be taken as the supremum of the set of permutations $\{t_i \sigma\}_{i\in B}$ in the Bruhat order.
\end{theorem}

\section{On combinatorial invariance}\label{sect:cic}

In the pursuit of proving the Combinatorial Invariance Conjecture, the expression found in \Cref{thm:hyper} was introduced in \cite{BBDVW21} as a stepping stone, particularly for symmetric groups. To revisit this theme, we will first recapitulate the classical conjecture, and then explore the implications of our findings on the conjecture through the lens of $J$-relative $R$-polynomials.

In this section $(W,S)$ is a Coxeter system with a finite $W$, and $T\subset W$ as before is the conjugation closure of $S$.

For a pair $\sigma, \omega\in W$, the \textit{Bruhat interval} $G[\sigma,\omega]= (V,E)$ is considered as a directed graph whose set of vertices $V$ is given by all elements $z\in W$ with $\sigma\leq z\leq \omega$, while the edges may be defined as ordered pairs of vertices
\[
E = \{(z_1,z_2)\in V\times V\;:\; z_2z_1^{-1}\in T, \ell(z_1) < \ell( z_2)\}\;.
\]

Naturally, an isomorphism of directed graphs $(V_1,E_1)$ and $(V_2,E_2)$ is a bijection $\phi:V_1\to V_2$, so that for a pair $z, z'\in V_1$, we have $(z,z')\in E_1$, if and only if, $(\phi(z),\phi(z'))\in E_2$.

The following is the much-studied \textit{combinatorial invariance conjecture}, which is attributed to Lusztig since the early 80's and to the thesis of Dyer \cite{Dye87}.

\begin{conjecture}\label{conj:cic}
For any Coxeter systems $(W_1,S_1)$, $(W_2,S_2)$ and pairs $\sigma_1,\omega_1\in W_1$, $\sigma_2,\omega_2\in W_2$ with Bruhat intervals $G[\sigma_1,\omega_1], G[\sigma_2,\omega_2]$ that are isomorphic as directed graphs, an equality $P_{\sigma_1,\omega_1} = P_{\sigma_2,\omega_2}$ holds.
\end{conjecture}

It is straightforward to verify the validity of the above conjecture would be equivalent to its validity with $P_{\sigma_i,\omega_i}$ replaced by either $\check{P}_{\sigma_i,\omega_i}$, $R_{\sigma_i,\omega_i}$ or $\check{R}_{\sigma_i,\omega_i}$.

We propose a following variant of combinatorial invariance which pertains the notion of $J$-relative $R$-polynomials.

\begin{conjecture}[Relative combinatorial invariance]\label{conj:R}
For any finite Coxeter system $(W,S)$, there exists an assigned subset $\mathbb{J}(W)\subsetneq S$ so that the following property holds.

Suppose that $\sigma_1,\omega_1\in W_1$, $\sigma_2,\omega_2\in W_2$ are two pairs in finite Coxeter systems $(W_1,S_1)$, $(W_2,S_2)$. Suppose that there is an isomorphism $\phi$ of Bruhat intervals $G[\sigma_1,\omega_1], G[\sigma_2,\omega_2]$ as directed graphs, such that for any edge $\sigma_1 \leq z < tz \leq \omega_1$ in $G[\sigma_1,\omega_1]$, the condition
\[
t\in (W_1)_{\mathbb{J}(W_1)},\;\mbox{ if and only if, }\; \phi(tz)\phi(z)^{-1} \in (W_2)_{\mathbb{J}(W_2)}
\]
holds.

Then, $R_{\sigma_1,\omega_1, \mathbb{J}(W_1)} = R_{\sigma_2,\omega_2, \mathbb{J}(W_2)}$.
\end{conjecture}

For a Coxeter system $(W,S)$, we say that a choice of a chain
\[
\underline{J} = (\emptyset = J_0 \subsetneq J_1 \subsetneq \ldots \subsetneq J_r = S )
\]
is a \textit{filtration} on the system. We call $(W, \underline{J})$ a \textit{filtered} Coxeter system.

A filtration $\underline{J}$ gives rise to a \textit{reflection-coloring} map $\alpha: T\to \{1,\ldots, r\}$ defined by setting $\alpha(t)$, for $t\in T$, to be the minimal $i$ for which $t\in W_{J_i}$ holds.

In the context of combinatorial invariance, we view $\alpha$ as a coloring, by at most $|S|$ colors, of edges in all Bruhat intervals $G[\sigma,\omega]$, for $\sigma,\omega\in W$.

\begin{conjecture}[Filtered combinatorial invariance]\label{conj:weak}
For any finite Coxeter system $(W,S)$, there exists an assigned canonical filtration $\underline{J}^0$, so that the following property holds.

Let $\sigma_1,\omega_1\in W_1$, $\sigma_2,\omega_2\in W_2$ be two pairs in finite Coxeter systems that are equipped with their canonical filtrations that give rise to corresponding reflection-coloring maps $\alpha_1,\alpha_2$.

Suppose that there is an isomorphism $\phi$ of Bruhat intervals $G[\sigma_1,\omega_1], G[\sigma_2,\omega_2]$ as directed graphs.

Suppose further that there is an increasing map $\gamma: \{1,\ldots,r_1\}\to \{1,\ldots,r_2\}$, such that for any edge $\sigma_1 \leq z < tz \leq \omega_1$ in $G[\sigma_1,\omega_1]$, the color identity $\gamma\circ \alpha_1(t) = \alpha_2\left( t'\right)$ holds, where $t'\in W_2$ is the reflection with $\phi(tz) = t'\phi(z)$.

Then, the equality $P_{\sigma_1,\omega_1} = P_{\sigma_2,\omega_2}$ holds.

\end{conjecture}

\begin{remark}
It is clear that \Cref{conj:weak} follows from the full combinatorial invariance as in \Cref{conj:cic}. In case the full conjecture is assumed, the notion of a canonical filtration becomes redundant and may be chosen arbitrarily.
\end{remark}

In what follows we will make use of the following basic lemma, whose proof we supply for completeness.

\begin{lemma}\label{lem:coset}
For $J\subset S$, suppose that $\sigma,\omega\in W$ are such that ${}^J\sigma= {}^J \omega$ (i.e. $W_J\sigma = W_J \omega$).

Then, the Bruhat interval $G[\sigma_J, \omega_J]$ for the Coxeter system $(W_J,J)$ is isomorphic to the Bruhat interval $G[\sigma,\omega]$ through the map $\kappa\mapsto \kappa \cdot{}^J \sigma$.
\end{lemma}

\begin{proof}
The map $z\mapsto {}^J z$ on $W$ preserves the Bruhat order (\cite[Proposition 2.5.1]{BB05}). Hence, for any $\sigma\leq z\leq \omega$, we must have ${}^J z = {}^J \sigma$.

In particular, it follows that for any $\sigma \leq z < tz \leq \omega$, we must have $t\in W_J$.

Finally, it remains to note that for any $t\in T\cap W_J$ and $\kappa\in W_J$ such that $\kappa < t\kappa$, we also have $\kappa\cdot{}^J\sigma < t\kappa \cdot{}^J\sigma$. Indeed, $\check{t}:=\kappa^{-1}t \kappa \in T\cap W_J$, $\kappa < \kappa \check{t}$ and ${}^J\sigma< \check{t}\cdot{}^J\sigma$.

\end{proof}

\begin{proposition}\label{prop:impl}
  Conjecture \ref{conj:R} implies Conjecture \ref{conj:weak}.
\end{proposition}

\begin{proof}
Suppose that relative combinatorial invariance holds, that is, for any finite Coxeter system $(W,S)$ a subset $\mathbb{J}(W)\subsetneq S$ is defined so that the condition in \Cref{conj:R} holds.

For such $(W,S)$ we assign a canonical filtration $\underline{J}^0 = \{J_i\}_{i=0}^r$ by taking $J_i = \mathbb{J}(W_{J_{i+1}})\subsetneq J_{i+1}$.

We will prove \Cref{conj:weak} by induction on the rank $|S|$ of the Coxeter system. As usual, it would suffice to prove equality of $R$-polynomials, since equality of Kazhdan-Lusztig polynomials would follow inductively from \eqref{eq:1}.

Now, suppose that pairs $\sigma_1,\omega_1\in W_1$, $\sigma_2,\omega_2\in W_2$, an isomorphism $\phi$ and a function $\gamma$ are given as is the assumptions of \Cref{conj:weak}.

If follows from the assumptions that no edges in $G[\sigma_2,\omega_2]$ are colored by a number greater than $\gamma(r_1)$. In other words, for all $t\in T_2$ and $z\in W_2$ with $\sigma_2\leq z < tz \leq \omega_2$, we have $\alpha_2(t) \leq \gamma(r_1)$. In particular, $\sigma_2\in W_{J_{\gamma(r_1)}}\omega_2$.


By \Cref{lem:coset}, we may identify $G[\sigma_2,\omega_2]$ with the Bruhat interval $G[(\sigma_2)_{J_{\gamma(r_1)}},(\omega_2)_{J_{\gamma(r_1)}}]$ in $W_{J_{\gamma(r_1)}}< W_2$. It is also easy to verify that $R_{\sigma_2,\omega_2} = R_{(\sigma_2)_{J_{\gamma(r_1)}},(\omega_2)_{J_{\gamma(r_1)}}}$ (For example, by taking the identity in \Cref{prop:recJR} and noting that $R_{\kappa \cdot {}^{J_{\gamma(r_1)}} \omega_2,\omega_2, J_{\gamma(r_1)}}=0$ according to \Cref{prop:Rformula}, for all $\kappa\in W_{J_{\gamma(r_1)}}$ unless $\kappa = (\omega_2)_{J_{\gamma(r_1)}}$.)

Thus, we may assume that $\gamma(r_1) = r_2$.

Let us denote $J = J_{r_1-1} = \mathbb{J}(W_1)$ and $J'= J_{r_2-1} = \mathbb{J}(W_2)$. For an edge $\sigma_1 \leq z < tz \leq \omega_1$, we see that the containment $t\in W_J$ is equivalent to $\alpha_1(t)\neq r_1$. By assumption on $\phi$, the last condition is also equivalent to $\alpha_2( \phi(tz)\phi(z)^{-1})\neq r_2$, which again means $\phi(tz)\phi(z)^{-1} \in W_{J'}$.

The assumed \Cref{conj:R} consequently implies that $R_{z, \omega_1, J} = R_{\phi(z), \omega_2, J'}$ holds, for all $\sigma_1\leq z\leq \omega_1$.

For a given $\sigma_1\leq z\leq \omega_1$ with ${}^J z= {}^J\sigma_1$, by \Cref{lem:coset} the Bruhat interval $G[\sigma_1,z]$ is isomorphic to $G[(\sigma_1)_J, z_J]$. When choosing a path in the graph, we can write $z= t_1\ldots t_s \sigma_1$ with $t_1,\ldots,t_s\in T_1\cap W_J$. Again, from the assumption on $\phi$ we deduce that there are $t'_1,\ldots, t'_s\in T_2 \cap W_{J'}$ for which $\phi(z) = t'_1\ldots t'_s \sigma_2$. Hence, ${}^{J'}\phi(z) = {}^{J'} \sigma_2$.

A similar argument on the inverse isomorphism $\phi^{-1}$ shows that in fact for any $\sigma_1\leq z\leq \omega_1$ with ${}^{J'}\phi(z) = {}^{J'} \sigma_2$ we will necessarily have ${}^J z= {}^J\sigma_1$.

By invoking \Cref{lem:coset} yet again on $G[\sigma_2,\phi(z)]$ and applying the induction hypothesis for the smaller rank Coxeter systems $(W_J,J)$ and $(W_{J'}, J')$, we deduce that $R_{(\sigma_1)_J,z_J} = R_{(\sigma_2)_{J'}, \phi(z)_{J'}}$.

Finally, by \Cref{prop:recJR},
\begin{equation*}
\begin{aligned}
R_{\sigma_1,\omega_1} &= \sum_{\kappa\in W_J\;:\; \sigma_1 \leq \kappa\cdot{}^J\sigma_1 \leq \omega_1 } R_{(\sigma_1)_J, \kappa} R_{\kappa\cdot {}^J\sigma_1, \omega_1, J}\\
 &= \sum_{\kappa\in W_J\;:\; \sigma_1 \leq \kappa\cdot{}^J\sigma_1 \leq \omega_1 } R_{(\sigma_2)_{J'}, \phi(\kappa{}^J\sigma_1)_J} R_{\phi(\kappa \cdot{}^J\sigma_1), \omega_2, J'}\\
&= \sum_{\kappa'\in W_{J'}\;:\; \sigma_2 \leq \kappa'\cdot{}^{J'}\sigma_2 \leq \omega_2 } R_{(\sigma_2)_{J'}, \kappa'} R_{\kappa' \cdot{}^J\sigma_2, \omega_2, J'}\\ &= R_{\sigma_2,\omega_2} \;.
\end{aligned}
\end{equation*}

\end{proof}

We also take a simple note that both our newly proposed conjectures are in fact weak variants of the full combinatorial invariance conjecture.

\begin{proposition}
  \Cref{conj:cic} implies \Cref{conj:R}.
\end{proposition}
\begin{proof}
The inverse relation in \Cref{prop:recJR} expresses $J$-relative $R$-polynomials in terms of $R$-polynomials for $W$ and $W_J$, as long as $W_J$ cosets in $W$ may be identified within the Bruhat intervals. Therefore, a similar argument as in the proof of \Cref{prop:impl} would suffice.
\end{proof}

To conclude this section, we will document how the hypercube decomposition presented in Theorem \ref{thm:wilmson}, as well as its elaboration in Section \ref{sect:Sn}, impact our recently suggested adaptations of combinatorial invariance.

\begin{theorem}[Filtered combinatorial invariance for symmetric groups]\label{thm:sncic}
Let $\sigma_1,\omega_1\in S_{n_1}$ and $\sigma_2,\omega_2\in S_{n_2}$ be two pairs of permutations.

Suppose that there is an isomorphism $\phi$ of Bruhat intervals $G[\sigma_1,\omega_1], G[\sigma_2,\omega_2]$ as directed graphs.

Suppose further that there is an increasing map $\gamma:\{1,\ldots,n_1\} \to \{1,\ldots,n_2\}$ such that the following property holds:

For any edge $\sigma_1\leq z < (i,j) z \leq \omega_1$ with $1\leq i<j\leq n_1$, so that $\phi((i,j) z) = (i',j')\phi(z)$ with $1\leq i' < j'\leq n_2$, we must have $j' = \gamma(j)$.

Then, the equality $P_{\sigma_1,\omega_1} = P_{\sigma_2,\omega_2}$ holds.
\end{theorem}

\begin{proof}
For the Coxeter system $(S_n, \{s_1,\ldots, s_n\})$, let us define $\mathbb{J}(S_n)= \{s_1,\ldots,s_{n-1}\}$. With this convention it is enough to prove \Cref{conj:R} for cases when the involved Coxeter systems are symmetric groups.

Indeed, the desired statement is \Cref{conj:weak} for those cases. It would follow after slightly unraveling the proof of \Cref{prop:impl} and noting that $W_{\mathbb{J}(W)}\cong S_{n-1}$ for $W\cong S_n$.

Now, \Cref{conj:R} for those case holds, as a consequence of \Cref{prop:normRj} and \Cref{lem:supr}, since they imply that $R_{\sigma,\omega, \mathbb{J}(S_n)}$ is an invariant of the Bruhat interval graph and a marking of the distinguished edges of the form $\sigma \leq z < (i,n)z \leq \omega$.

\end{proof}

\bibliographystyle{alpha}
\bibliography{ref}

\end{document}